\documentclass[11pt]{article}%
\usepackage[english]{babel}
\usepackage{amsfonts, amsmath, amssymb,amscd}
\usepackage[all]{xy}
\usepackage{latexsym}
\usepackage{amsmath, amsthm}
\usepackage{amsfonts}
\usepackage{amssymb}
\usepackage{amsmath}
\usepackage{graphicx,color}
\newtheorem{theorem}{Theorem}[section]
\newtheorem{lemma}[theorem]{Lemma}
\newtheorem{proposition}[theorem]{Proposition}

\newtheorem{corollary}[theorem]{Corollary}

\newcommand{\lap}{\Delta}

\newcommand{\R}{\mathbb{R}}
\newcommand{\Hi}{\mathbb{H}}

\newcommand{\grad}{\operatorname*{grad}}

\begin{document}

\pagestyle{myheadings}

\markboth{LEONARDO BONORINO AND PATR\'{I}CIA KLASER}{EXISTENCE AND NONEXISTENCE OF EIGENFUNCTIONS}

\title{Existence and nonexistence results for eigenfunctions of the Laplacian in unbounded domains of $\mathbb{H}^n$}
\author{Leonardo Prange Bonorino
\and Patr\'{\i}cia Kruse Klaser  }
\date{}
\maketitle

\begin{abstract}
We investigate, for the Laplacian operator, the existence and nonexistence of eigenfunctions of eigenvalue between zero and the first eigenvalue of the hyperbolic space $\mathbb{H}^n$, for unbounded domains of $\mathbb{H}^n$. If a domain $\Omega$ is contained in a horoball, we prove that there is no positive bounded eigenfunction that vanishes on the boundary.  However, if the asymptotic boundary of a domain contains an open set of $\partial_{\infty}\mathbb{H}^n$, there is a solution that converges to $0$ at infinity and can be extended continuously to the asymptotic boundary. In particular, this result holds for hyperballs.  \end{abstract}


\section{Introduction}
The first eigenvalue of a noncompact Riemannian manifold $M$ is defined by
$$\lambda_1(M)=\inf\{\lambda_1({\mathcal O})\; \big{|} \; {\mathcal O} \text{ is a bounded domain of }M\},$$
where $\lambda_1({\mathcal O})$ is the first eigenvalue of the Laplacian on ${\mathcal O}$. 
We say that $\lambda \in \mathbb{R}$ is an eigenvalue of the Laplacian in $\Omega\subset M$ if there is some nontrivial function $u \in C(\overline{\Omega}) \cap C^2(\Omega)$ such that 
\begin{equation}\label{eqProblemaPrincipal}
\left\{
\begin{array}{rclc}
- \Delta u & = & \lambda \, u & {\rm in } \; \;\Omega \\[5pt]
  u & = & 0 & {\rm on } \; \; \partial \Omega. \\ \end{array}
\right. 
\end{equation}
In this case, we say that $u$ is an eigenfunction associated to $\lambda$ in $\Omega$. We do not require that $u$ is bounded nor converges to $0$ at infinity.

In the special case $M=\mathbb{H}^n$, McKean proved in \cite{Mk} that $$\lambda_1:=\lambda_1(\Hi^n)=\frac{(n-1)^2}{4}.$$

In this work we study existence and nonexistence of eigenfunctions in subsets of the hyperbolic spaces $\mathbb{H}^n$ associated to $\lambda\in[0,\lambda_1]$, where $\lambda_1$ is the first eigenvalue of $\Hi^n.$
Moreover we investigate when such eigenfunctions converges to zero at infinity. 
 
It is interesting to observe that for the Euclidean space, $\lambda_1(\mathbb{R}^n)=0$ and the equivalent problem consists in finding harmonic functions that vanish on the boundary. Observe that for noncompact domains in $\mathbb{R}^n$, any bounded harmonic function that vanishes on the boundary and at infinity is trivial. For $\mathbb{H}^n$, if $\lambda=\lambda_1(\mathbb{H}^n)$, there exist nonconstant bounded eigenfunctions of eigenvalue $\lambda$, defined in the whole $\mathbb{H}^n$, that vanish at infinity. Indeed, Grellier and Otal \cite{GO} gave an integral expression  to the radial eigenfunctions associated to any $\lambda \in (0,\lambda_1]$, for $n \geq 3$. For $n=2$, a similar characterization can be found in \cite{O}. We notice that these eigenfunctions are decreasing and bounded above by $Cre^{-\sqrt{\lambda}r}.$

It is natural to wonder what domains in $\Hi^n$ admit eigenfunctions associated to $\lambda\in[0,\lambda_1]$, that not necessarily converge at infinity. 
Another point of interest is to determine whether an eigenfunction in $\Omega \subset \mathbb{H}^n$ can be extended continuously to the asymptotic boundary of $\Omega,$ as the zero function. We consider $\overline{\mathbb{H}^n}=\mathbb{H}^n\cup \partial_\infty\mathbb{H}^n$, where $\partial_\infty\mathbb{H}^n$ is the asymptotic boundary of $\mathbb{H}^n$ with the cone topology (see \cite{EO}). The asymptotic boundary of a subset $A$ is defined by $\partial_{\infty}A=\bar{A} \cap \partial_{\infty}\mathbb{H}^n$, where $\bar{A}$ is the closure of $A$ in $\overline{\mathbb{H}^n}$.

The study of these two questions is the main purpose of this work. The answer to the first one is related to the question ``how large is the asymptotic boundary of this domain?''.
We prove in Section \ref{secaoHorobolas} that for any domain $\Omega$ contained in a horoball, it does not exist a bounded eigenfunction in $\Omega$ associated to $\lambda\in [0,\lambda_1]$.  This nonexistence result includes the harmonic case. In Section \ref{generalBigDomains}, we show the existence of a positive bounded eigenfunction associated to $\lambda \in (0,\lambda_1]$ in $\Omega$ if $\partial_\infty \Omega$ has nonempty interior in $\partial_{\infty}\Hi^n.$ 

Concerning to the second question, we present in Section \ref{generalBigDomains} a positive bounded eigenfunction defined outside a horoball that cannot be extended continuously to the asymptotic boundary. On the other hand, we obtain in this section eigenfunctions that have a continuous extension to the asymptotic boundary, where it vanishes. We prove also that if an eigenfunction converges as $x \to \infty$, then the limit must be zero.

These results lead us to the conclusion that the topology of the asymptotic boundary of the domain is essential for the existence and nonexistence of bounded eigenfunctions. In fact, the open subsets of $\overline{\Hi^n}$ that intercept  $\partial \Hi^n$ admit a bounded eigenfunction that vanishes at infinity, but the horoball (or any subset of it), that has at most one point of $\partial_\infty\Hi^n$, does not admit solution.
Also we must point out that some domains, like the horoball, do not admit bounded eigenfunctions but they admit unbounded ones. 

This work is organized as follows. In Section \ref{preliminaries}, we review some comparison principles. Some useful results of \cite{GO}, about radially symmetric eigenfunctions, are studied in Section \ref{secaoHorobolas}, where we also show some properties of radial eigenfunctions defined outside a ball. In Section \ref{secaoHorobolas} and \ref{generalBigDomains}, we prove the main results regarding existence and nonexistence of eigenfunctions in unbounded domains.

\

\noindent {Acknowledgments.}
We would like to thank Prof. Dr. Jaime Bruck Ripoll by his support and useful suggestions to this research work.

\ 

\section{Preliminaries}\label{preliminaries}

We present two comparison lemmas about eigenfunctions in Riemannian manifolds, that will be needed in the next sections.

\begin{lemma}\label{comparison}
(Comparison Principle I) Let $\Omega$ be a bounded domain and $u, v: \bar{\Omega} \rightarrow \R$ be nonnegative solutions of
$$-\lap w=\lambda w,$$
where $\lambda \ge 0$.
If $u\geq v>0$ on $\partial \Omega,$ then $u\geq v$ in $\Omega.$
\end{lemma}

\begin{proof}
Suppose that there is $x_0\in \Omega$ such that $v(x_0)>u(x_0).$
Define $$C=\sup\{c\geq 0 \; | \; u\geq c\, v \; {\rm in } \; \Omega\}.$$
It is clear that $C<1.$ Hence $u\geq v > Cv$ on $\partial \Omega.$
From the definition of $C,$ $u-Cv\geq 0$ in $\Omega.$

We claim that $u-Cv$ has an interior minimum. Observe that the minimum of $u-Cv$ is zero and it cannot happen in $\partial\Omega$ because $u-Cv > u-v \geq 0$ on $\partial \Omega.$

On the other hand, $u-Cv$ is superharmonic because
$$-\lap(u-Cv)=\lambda (u-Cv)\geq 0$$
Hence it cannot have a interior minimum and we have a contradiction.
\end{proof}

\begin{lemma}\label{comparison II}
(Comparison Principle II) Let $\Omega$ be a bounded domain, $u \in C(\overline{\Omega})\cap C^2(\Omega)$ be a supersolution, and $v\in C(\overline{\Omega})\cap C^2(\Omega)$ be a subsolution of 
$$-\lap w=\lambda w,$$
where $\lambda<\lambda_1(\Omega).$
If $u > v$ on $\partial \Omega,$ then $u > v$ in $\Omega.$
\end{lemma}

\begin{proof}
Suppose that $\Omega':=\{x\in\Omega\;|\; v(x)>u(x)\}$ is not empty. Then $w_0:=v-u$ is a positive function in $\Omega'$ that satisfies $-\Delta w_0 \le \lambda w_0$, vanishes on $\partial \Omega' \subset \Omega$, and belongs to $H_0^1(\Omega')$.  Multiplying this inequality by $w_0$ and using the divergence theorem, we get
$$ Q(w_0):= \frac{\displaystyle \int_{\Omega'} |\nabla w_0|^2 dx }{\displaystyle \int_{\Omega'} |w_0|^2 dx } \le \lambda $$
But this contradicts $Q(w_0) \ge \inf_{w\in H_0^1(\Omega')} Q(w)= \lambda_1(\Omega')\geq\lambda_1(\Omega)>\lambda$. Hence $v \le u$ and, therefore, $-\Delta v \le -\Delta u$ in $\Omega$.
This, the Strong Maximum Principle and $u > v$ on $\partial \Omega$ imply that $u > v$ in $\Omega$.  
\end{proof}

\begin{corollary}
If $\lambda \le \lambda_1$ and $\Omega$ is a bounded domain, then there is no eigenfunction associated to $\lambda$ in $\Omega$.
\label{NoEigenfunctionInBoundedDomains}
\end{corollary}

\begin{proof}
For a bounded domain $\Omega$, let $B$ be an open ball that contains $\overline{\Omega}$, $\lambda_B$ be the first eigenvalue associated to $B$ and $u$ be a positive eigenfunction associated to $\lambda_B$. Then $\lambda_B > \lambda_1$ from
\cite{Mk} and, therefore, $\lambda_B > \lambda$. Hence if $v$ is a solution of $-\Delta w = \lambda w$ in $\Omega$ that vanishes on $\partial \Omega$, it follows from Lemma \ref{comparison II} that $v \le \alpha u$ for any $\alpha > 0$, since $\alpha u > 0= v$ on $\partial \Omega$. Therefore, $v \le 0$. Applying the same argument for $-v$, we get that $v\ge 0$, concluding the result.
\end{proof}

\section{Nonexistence results}
\label{secaoHorobolas}
The main purpose of this section is to prove that if a domain is contained in a horoball, then there is no bounded eigenfunction associated to $\lambda \in [0,\lambda_1]$ in this domain.
For that, we need to define some barriers, which are positive radial eigenfunctions associated to $\lambda$ in the complement of a ball.

First observe that the problem $$\left\lbrace\begin{array}{l}
            -\lap u=\lambda u \text{ in }\mathbb{H}^n\\[5pt]
            u\geq 0\text{ is a bounded radially symmetric function around } o\\
            \end{array}\right.$$
 has solutions for $0 < \lambda \le \lambda_1$, which are presented as integral formulas in \cite{GO}. There, the authors also exhibit an unbounded function defined in $\mathbb{H}^n\backslash\{o\}$, that correspond to the Greens function, 
and show that both solutions converge to $0$ at infinity. Indeed the following two results are special cases of those proved in \cite{GO}.

\begin{lemma}\label{autofuncao global}
For $\lambda \in (0,\lambda_1]$ and $o \in \mathbb{H}^n$, there exists a bounded eigenfunction $u \in C^2(\mathbb{H}^n)$ associated to $\lambda$ in $\mathbb{H}^n$, radially symmetric with respect to $o$. Furthermore, there exists also an eigenfunction $v \in C^2(\mathbb{H}^n \backslash \{o\})$, radially symmetric with respect to $o$ and that has a singularity at this point. 
\end{lemma}

\begin{lemma}
\label{convergesToZeroAtInfinity}
If $u$ and $v$ are the eigenfunctions presented in the previous lemma, then $u(x)$ and $v(x)$ converge to $0$ as $d(x,o) \to \infty$.
\end{lemma}

These lemmas imply the next result, for which we need to remind that
if $r(x)$ is the distance between $x$ to some fixed point in $\mathbb{H}^n,$ then
\begin{equation} \label{EDOpponto}
\lap\left(u\circ r\right)(x)=u''(r)+(n-1)\coth(r)u'(r),
\end{equation}
for any $C^2$ function $u:[0,+\infty) \to \mathbb{R}$. This expression can be proved by choosing $g(x)=r(x)$ in the relation
\begin{equation}
\lap\left(u\circ g\right)(x)=u''(g(x))|\grad g(x)|^{2}+u'(g(x))\lap g(x),
\label{LaplacianoCompostasUg}
\end{equation}
that holds for any $C^2$ functions $g:\mathbb{H}^n \rightarrow I$ and $u: I \rightarrow \R$,  where $I \subset \mathbb{R}$ is some interval.

\begin{proposition}
For $o \in \mathbb{H}^n$, $R > 0$ and $\lambda \in (0,\lambda_1]$, there exists a positive bounded radially symmetric eigenfunction $\bar{u}$ associated to $\lambda$ in $\mathbb{H}^n\backslash \overline{B_R(o)}$. Moreover,
$\bar{u}(x) \to 0$ as $d(x,o) \to \infty$.
\label{eigenfunctionOutsideABall}
\end{proposition} 

\begin{proof}
Let $u$ and $v$ be the radially symmetric eigenfunctions presented in Lemma \ref{autofuncao global}. According to \eqref{EDOpponto}, they are solutions of
$$ w''(r)+(n-1)\coth(r)w'(r)+ \lambda w(r)=0,$$
where $r=dist(x,o)$ and the prime symbol denote the derivation with respect to $r$. Since $u$ and $v$ are linearly independent, there exist $\alpha$ and $\beta$ reals 
such that $\alpha u(R) + \beta v(R) = 0$ and $\alpha u + \beta v \not\equiv 0$. Therefore $\bar{u} := \alpha u + \beta v$ satisfies $-\Delta \bar{u} = \lambda \bar{u}$, $\bar{u}=0$ on $\partial B_R(o)$, and $\bar{u}$ is radially symmetric with respect to $o$. Since $u$ and $v$ converges to zero at infinity, the same holds for $\bar{u}$. Hence $\bar{u}$ is bounded in $\mathbb{H}^n\backslash B_R(o)$. Furthermore $\bar{u}$ cannot change sign in $\mathbb{H}^n\backslash B_R(o)$, otherwise $\bar{u}$ would be an eigenfunction associated to $\lambda$ for some bounded annulus $A$, yielding a contradiction with Corollary \ref{NoEigenfunctionInBoundedDomains}. Hence we can suppose that $\bar{u}$ is positive in $\mathbb{H}^n\backslash \overline{B_R(o)}$.
\end{proof}

Finally we obtain some properties of radial eigenfunctions defined in $\mathbb{H}^n$ or outside a ball. 
\begin{lemma}\label{exiteautofuncao}
Suppose that $u$ is a positive bounded eigenfunction in $\mathcal{O}=\mathbb{H}^n$ or $\mathcal{O}=\mathbb{H}^n\backslash B_R(o)$, radially symmetric with respect to $o$.
\\ (i) If $\mathcal{O}=\mathbb{H}^n$, then $u(r)$ is a decreasing function of $r=dist(x,o)$ and does not admit any critical point.
\\ (ii) If $\mathcal{O}=\mathbb{H}^n\backslash B_R(o)$, then there exists $R_0$ such that $u(r)$ is increasing in $[R,R_0]$ and decreasing in $[R_0,+\infty)$.
\end{lemma}
\begin{proof} (i) If $u$ is not decreasing with respect to the radius $r(x)=dist(x,o)$, then either $o$ is the minimum point of $u$ in some ball $B_R(o)$ or there exists some open annulus $\{ x \, : \, R_1 < r(x) < R_2\}$ that has points of local minimum of $u$. In both cases, we have a contradiction with the maximum principle and the fact that $u$ is superharmonic.  
\\ (ii) Since $u(R)=0$ and $\lim_{r\to +\infty} u(r)=0$ according to Lemma \ref{convergesToZeroAtInfinity}, then there exists some point of local maximum $R_0 > R$.
If there are other points of local maximum, then we can find some open annulus that has points of local minimum of $u$, yielding a contradiction as in the proof of (i).
\end{proof}

\subsection{Nonexistence result for domains of a horoball}

A horosphere $H$ determines two noncompact sets, one of them is a horoball $B$ and the other, $B^c$, corresponds to
the exterior of $B$.
If $g(x)=d(x)=dist(x,H)$ is the distance from $x$ to $H=\partial B$, then relation \eqref{LaplacianoCompostasUg} implies that
\begin{equation} \label{EDO horo interior}
\lap\left(u\circ d\right)(x)=u''(d)-(n-1)u'(d) \text{ if } x\in B
\end{equation}
and
\begin{equation} \label{EDO horo exterior}
\lap\left(u\circ d\right)(x)=u''(d)+(n-1)u'(d) \text{ if } x\notin B.
\end{equation}
Hence the eigenfunctions defined in the horoball $B$, of the form 
$v=v(d(x))$, satisfy 
$$ v''(d)-(n-1)v'(d) + \lambda v(d) =0 \quad {\rm and } \quad v(0)=0,$$
which is solved by multiples of

\begin{equation} v(d)= \left\{ \begin{array}{ll} \,d \,e^{\frac{(n-1)d}{2}} & \quad {\rm if} \quad \lambda = \lambda_1  \\[5pt] e^{\frac{(n-1)d}{2}}\left( e^{\frac{ \sqrt{(n-1)^2 -4\lambda}}{2}d} - e^{-\frac{ \sqrt{(n-1)^2 -4\lambda}}{2}d} \right) & \quad {\rm if} \quad \lambda < \lambda_1 \end{array} \right.
\label{vUnboundedEigenfucntionInHoroball}
\end{equation}
We refer to the function $v$ as the usual eigenfunction of the horoball. Observe that it does not change sign and is not bounded.

From now on we assume that $\Omega$ is an unbounded open set of $\mathbb{H}^n$, otherwise, from Corollary \ref{NoEigenfunctionInBoundedDomains}, there is no eigenfunction associated to $\lambda \le \lambda_1$.

\begin{lemma}\label{lemma do C0 e d0}
There exist positive constants $d_0$ and $C_0$, that depend only on $n$, such that if $\Omega$ is an open subset of a horoball $B$ and $u$ is a bounded eigenfunction associated to $\lambda \in [0,\lambda_1]$ in $\Omega$, then 
$$ |u(x)| \le C_0 \; dist(x,\partial B) \sup_{\Omega} |u| \quad {\rm if } \quad dist(x,\partial B) \le d_0 .$$
The same inequality holds if $\Omega \subset \mathbb{H}^n \backslash B$.
\label{controleDeUPelaDistanciaAFronteira}
\end{lemma}

\begin{proof}
For any ball $B_1$ of radius one, from Proposition \ref{eigenfunctionOutsideABall}, there exists a radially symmetric eigenfunction $w$ associated to $\lambda_1$ in $\mathbb{H}^n \backslash B_1$. Moreover, defining $r(x)$ as the distance between $x$ and the center of $B_1,$ (ii) of Lemma \ref{exiteautofuncao} implies that there exists $R_0 >1$, depending only on $n$, such that $w(r)$ is increasing in $[1,R_0]$ and decreasing in $[R_0,+\infty)$. By normalizing, we can assume that $w(R_0)=1$ and, from its regularity, $w$ is a Lipschitz function in $\overline{B_{R_0}\backslash B_1}$. We denote the Lipschitz constant of $w$ in this set by $C$. Observe that $R_0$ and $C$ are the same for any ball of radius one, since $\mathbb{H}^n$ is homogeneous. Define $d_0=R_0-1$.

Given $x_0 \in \Omega$ such that $dist(x_0,\partial B) < d_0$, let $x_1 \in \partial B$ that satisfies $dist(x_0,x_1)=dist(x_0,\partial B)$.
For positive $\epsilon < d_0 - dist(x_0, \partial B) $, consider the ball $B_1(o) \subset \mathbb{H}^n \backslash B$ centered at $o$ with radius one, such that $x_1 \in \partial B_{1+\epsilon}(o)$. 
Then $d(x_1,o)=1+\epsilon$ and 
\begin{equation} d(x_0,o) \le d(x_0,x_1)+ d(x_1,o)=d(x_0,x_1) + 1+\epsilon < 1+d_0 = R_0.
\label{distanciaX0o}
\end{equation} 
Let $w$ be the positive radial eigenfunction associated to $\lambda_1$ in $\mathbb{H}^n \backslash B_1(o)$, as we mentioned previously. 
Now define the barrier
$$ \bar{w}(x) =  2 (\sup_{\Omega}|u|) w(x). $$
We can prove that $\bar{w} > |u|$ in the closure of $\Omega \cap B_{R_0}(o)$. For this observe that $\partial (\Omega \cap B_{R_0}(o)) \subset \partial \Omega \cup \partial B_{R_0}(o) $. On $\partial B_{R_0}(o)$, $w(x)=\max|w|=1$ and, therefore, $\bar{w}=2 \sup|u| > |u|$ on $\partial B_{R_0}(o)$. Moreover, $\bar{w} > |u|=0$ on $\partial \Omega$, since $\bar{w} > 0$ in $\mathbb{H}^n \backslash \overline{B_{1}(o)} \supset \mathbb{H}^n \backslash B_{1+\epsilon}(o) \supset \overline{\Omega} \supset \partial \Omega$. Then, it follows that $\bar{w} > |u|$ on $\partial (\Omega \cap B_{R_0}(o))$.
We have also, from Corollary \ref{NoEigenfunctionInBoundedDomains}, that $\lambda_1(\Omega \cap B_{R_0}(o)) > \lambda_1 \ge \lambda$, since $\Omega \cap B_{R_0}(o)$ is bounded. Hence, Lemma \ref{comparison II} implies that $\bar{w} > |u|$ in $\overline{\Omega \cap B_{R_0}(o)}$.

Therefore, using \eqref{distanciaX0o}, we obtain $|u(x_0)| \le \bar{w}(x_0)$. Observe that $\bar{w}$ is a Lipschitz function in $B_{R_0}(o)\backslash B_{1}(o)$ and its Lipschitz constant is $2C \sup|u|$. Hence, for $x_2 \in \partial B_1(o)$ such that $dist(x_2,x_1)=\epsilon$, 
 \begin{align*} |u(x_0)| \le \bar{w}(x_0)= \bar{w}(x_0)-\bar{w}(x_2) &\le 2C \sup |u| dist(x_0,x_2) \\[5pt]
                                                                       &\le 2C \sup |u| ( dist(x_0,x_1) + \epsilon). \end{align*}
Since $\epsilon >0$ can be taken arbitrarily small and $dist(x_0,x_1)=dist(x_0,\partial B)$, defining $C_0=2C$, the result follows if $dist(x_0,\partial B) < d_0$.
From the continuity of $u$, the result also holds if $dist(x_0,\partial B) = d_0$. In the case $\Omega \subset \mathbb{H}^n \backslash B$, the argument is the same.
\end{proof}

Given a horosphere $H$ that is the boundary of a horoball $B$, define the horospheres with respect to $B$ by
$$H_d=\left\lbrace \begin{array}{l}
                 \{x\in B \; |\; dist(x,H)=d\} \text{ if }d > 0\\[5pt]
                 \{x\notin B \; | \; dist(x,H)=-d\} \text{ if }d\le 0
                 \end{array}\right.$$
The horoannulus $A_{a,b}$ is the domain bounded by $H_a$ and $H_b,$ that is,
$$A_{a,b}=\{x\in \Hi^n \; |\; x\in H_d \text{ for some }d \in (a,b)\}.$$

From the previous lemma, we obtain the next result.

\begin{theorem}
Consider $C_0$ and $d_0$ as in Lemma \ref{lemma do C0 e d0}. Let $C > C_0$ and $B$ be a horoball. Suppose that either $\Omega \subset A_{0,b}$ or $\Omega \subset A_{-b,0}$ for some positive $b \le \min \{ 1/C, d_0 \}.$ Then there is no bounded eigenfunction associated to $\lambda \in [0,\lambda_1]$ in $\Omega$.
\label{nonexistenceIfWidthIsSmall}
\end{theorem}

\begin{lemma}
Let $B$ be a horoball in $\mathbb{H}^n$ and $d_1=\frac{1}{2}\min \{ 1/C_0, d_0 \}$. There exists $0< \delta \le d_1$, that depends only on $n$, such that if $u$ is a bounded eigenfunction associated to $\lambda$ in $\Omega \subset A_{0,d}$, where $d > d_1$, then there is a positive bounded eigenfunction $\tilde{u}$ associated to $\lambda$ in some $\tilde{\Omega} \subset A_{0,d - \delta}$.
\label{solutionInASmallerHoroannulus}
\end{lemma}

\begin{proof}
For the horosphere $H_{-1}$ with respect to $B$, consider the function $\tilde{v}(x)=v(dist(x, H_{-1}))$, where $v$ is given by \eqref{vUnboundedEigenfucntionInHoroball}. Observe that $\tilde{v}$ is an eigenfunction associated to $\lambda$ in the horoball $B'$ bounded by $H_{-1}$.
Let $$v_0=\frac{\tilde{v}}{v(d+2)} \quad {\rm and} \quad  u_0=\frac{u}{\sup |u|}. $$
We prove that $\tilde{u}:=(u_0 - v_0)^+$, restricted to the right domain, is the function we are looking for.
First, since $\overline{\Omega} \subset B'$, $u_0(x) - v_0(x) < 0$ for $x \in \partial \Omega$ and, therefore, ${\rm supp} \; \tilde{u} \subset \Omega$. Observe that $1 \le dist(x, H_{-1}) \le d+1$ for $x \in A_{0,d}$ and $v$ is increasing in $d$, then $v(1) \le \tilde{v}(x)\le v(d+1)$ in $A_{0,d}$. Therefore, $$\gamma_1 := \frac{v(1)}{v(d+2)} \le v_0(x) \le \gamma_2:= \frac{v(d+1)}{v(d+2)} < 1 \quad {\rm  for} \quad x\in A_{0,d}.$$
Hence, using that $\sup u_0 = 1 > \gamma_2$, it follows that the set $\{ u_0 > v_0\}$ is nonempty and, therefore, ${\rm supp} \; \tilde{u} \ne \emptyset$. Let $\tilde{\Omega}$ be a subdomain of $\Omega$ such that $\tilde{u}=0$ on $\partial \tilde{\Omega}$ and $\tilde{u} > 0$ in $\tilde{\Omega}$. To complete the proof, we show that $\tilde{\Omega} \subset A_{0,d-\delta}$ for $\delta= \frac{1}{2}\min \{ d_0, \frac{\gamma_1}{C_0}\} \le d_1$, that depends only on $n$. Since $\tilde{\Omega}\subset A_{0,d}$, we just need to prove that $dist(x,H_d) > \delta$ for $x \in \tilde{\Omega}$. Suppose that $x_1 \in \tilde{\Omega}$ and $dist(x_1,H_d) \le d_0$. (The case $dist(x_1,H_d) > d_0$ is trivial since $d_0 > \delta$.) Then $u_0(x_1) > v_0(x_1) \ge \gamma_1$. From Lemma \ref{controleDeUPelaDistanciaAFronteira} and $\sup u_0 =1$, we have 
$$ \gamma_1 < u_0(x_1) \le C_0 \; dist (x_1, H_d), $$
concluding that $ dist (x_1, H_d) > \frac{\gamma_1}{C_0}\ge \delta$.
\end{proof}
\begin{theorem}
Let $B$ be a horoball in $\mathbb{H}^n$, $\Omega \subset B$ and $\lambda \in [0,\lambda_1]$.
Then there is no bounded eigenfunction associated to $\lambda$ in $\Omega$. 
\end{theorem}

\begin{proof}
Let us assume that there is a bounded eigenfunction $u_0$ associated to $\lambda$ in $\Omega$. Consider the positive eigenfunction $\bar{v}$ in $B$ associated to $\lambda$ defined by $\bar{v}(x)= v(dist(x, \partial B))$, where $v$ is given by  \eqref{vUnboundedEigenfucntionInHoroball}. Observe that for some constant $C$, the set $\Omega_1 = \{ x \in \Omega \; : \; C u_0(x) > \bar{v}(x)\}$ is not empty.  On the other hand, since $C u_0$ is bounded, say by $M$, and $v(dist(x,\partial B)) \to +\infty$ as $dist(x,\partial B) \to +\infty$, then 
$$\bar{v}(x) = v(dist(x,\partial B)) > M \ge C u_0(x) \quad {\rm if} \quad dist(x,\partial B) \ge \bar{d},$$
where $\bar{d}$ is sufficiently large. 
Hence $\Omega_1 \subset A_{0,\bar{d}}$ and $u_1$ defined in $\overline{\Omega}_1$ by $u_1= Cu_0 -\bar{v}$ is zero on $\partial \Omega_1$. Indeed, $u_1$ is a positive bounded eigenfunction associated to $\lambda$ in $\Omega_1$. If $\bar{d} \le d_1= \frac{1}{2}\min \{ 1/C_0,d_0 \}$, any subdomain of $A_{0,\bar{d}}$ cannot have a bounded eigenfunction associated to $\lambda$ according to Corollary \ref{nonexistenceIfWidthIsSmall}. This contradicts the existence of $u_1$. Hence $\bar{d} > d_1$ and, from Lemma \ref{solutionInASmallerHoroannulus}, there exists a positive bounded eigenfunction $u_2$ associated to $\lambda$ in some $\Omega_2$ contained in $A_{0,\bar{d}- \delta}$.

If $\bar{d}-\delta \le d_1$, we have a contradiction as before. Then $\bar{d}-\delta > d_1$ and we can apply Lemma \ref{solutionInASmallerHoroannulus} to obtain a bounded eigenfunction $u_3$ in some $\Omega_3$ contained in $A_{0,\bar{d}-2 \delta}$.
Indeed, since there exists some $k \in \mathbb{N}$ such that $0 < \bar{d}-k\delta \le d_1$, applying again Lemma \ref{solutionInASmallerHoroannulus} $(k-2)$ times, it follows that there is some positive bounded eigenfunction $u_{k+1}$ in some $\Omega_{k+1}$, subdomain of $A_{0,\bar{d}-k \delta}$.  However this contradicts Corollary \ref{nonexistenceIfWidthIsSmall}, since any subdomain of  $A_{0,\bar{d}-k \delta}$ cannot have a bounded eigenfunction associated to $\lambda \in [0,\lambda_1]$.  
\end{proof}

\noindent {\sl Remark:} In particular, this result holds for the harmonic case, that is, there is no nontrivial bounded harmonic functions defined in subdomains of a horoball, vanishing on the boundary.

\section{Existence results}
\label{generalBigDomains}

In this section we study the existence of nonnegative bounded eigenfunctions that admit a continuous extension to $\partial_\infty \mathbb{H}^n.$
We show that if the asymptotic boundary of a domain contains an open set of $\partial_{\infty}\mathbb{H}^n$, then this domain
admits such an eigenfunction. Moreover, we demonstrate that if such extension is possible for some eigenfunction $u$ associated to $\lambda \in (0,\lambda_1]$, then $\lim_{z \to z_0}u(z)=0$ for any $z_0$ in the asymptotic boundary. 

Recall that a hypersphere is a hypersurface equidistant from a totally geodesic hypersurface of $\Hi^n$ and a hyperball is a connected component of the complement of a hypersphere. Hence a domain contains a hyperball if and only if its asymptotic boundary contains an open set of $\partial_{\infty}\mathbb{H}^n.$

\begin{theorem}
Let $\Omega$ be an open set in $\mathbb{H}^n$ that satisfies the exterior sphere condition, that is, for any $x\in \partial \Omega$ there exists a geodesic ball $B$ such that $\overline{\Omega} \cap \bar{B} =\{ x \}$. Assume also that $\partial_{\infty} \Omega$ contains an open subset of $\partial_{\infty}\mathbb{H}^n.$ Then, for any $\lambda \in (0,\lambda_1]$, $\Omega$ admits a positive bounded eigenfunction of eigenvalue $\lambda$.
\end{theorem}

\begin{proof}
Since $\partial_{\infty} \Omega$ contains some open subset of $\partial_\infty \Hi^n,$ there is a hyperball $H$ in $\Omega$.
We may assume that $\partial H$ is totally geodesic.

Let $p_1$ and $p_2$ be two points that are equidistant from $\partial H,$ such that $p_1\in H$, $p_2 \notin H$, and the geodesic that connects these two points intercepts $\partial H$ orthogonally. From \cite{GO}, there exists a global radially symmetric eigenfunction $v_i$ associated to $\lambda$, centered at $p_i$, $i=1,2$. We can suppose that $v_i(p_i)=1.$ Define $$v_0=\left\lbrace \begin{array}{ll}
                                        v_1-v_2 & \text{ in }H\\[5pt]
                                        0 & \text{ in }H^C.
                                        \end{array}\right.$$
Since $v_1=v_2$ on $\partial H,$ $v_0$ is a continuous function and it follows from $v_i$ being a decreasing function of the distance to $p_i$ that $v_0\geq 0.$

Let $x_0$ be a point in $\partial \Omega$ and consider the problem 
$$(P_R) \quad \quad \quad  \left\lbrace \begin{array}{cl}
-\lap u= \lambda u & \text{ in }\Omega_R=\Omega \cap B_R(x_0)\\[5pt]
u=v_0 & \text{ on }\partial \Omega_R \\
\end{array}\right.$$
for $R$ large enough such that $H\cap \Omega_R\neq \emptyset$. As an application of the Fredholm alternative (see Theorem 8.6 in \cite{GT}), this problem has a solution, which is unique, since Corollary \ref{NoEigenfunctionInBoundedDomains} implies that $\lambda$ is not in the spectrum of $-\Delta$ for bounded domains. 
Let $u_N$ be the solution of $(P_N),$ $N\in\mathbb{N}.$ Since $u_N\geq v_0$ on $\partial \Omega_N,$ by the comparison principle (Lemma \ref{comparison II}), $u_N\geq v_0$ in $\Omega_N.$ On the other hand $u_N=v_0 \leq v_1$ on $\partial \Omega_N,$ since $v_2\ge 0$. Hence $u_N \le v_1$ in $\Omega_N.$ We conclude
\begin{equation}\label{desigualdades}
0\leq v_0\leq u_N \leq v_1 \text{ in } \Omega_N.
\end{equation}

Therefore $u_N$ is a bounded sequence and, using that $-\lap u_N= \lambda u_N$, it follows from classical estimates that the derivative of $u_N$ is uniformly bounded in $\Omega_{N_0}$ for $N\ge N_0$. Hence for any $x \in \Omega$,  $\{u_N\}_{N \ge N_0}$ is an equicontinuous family at $x$, where $N_0$ is large enough. Then there exists a subsequence converging uniformly on compact subsets of $\Omega$ to a solution $u$ of
$$\left\lbrace \begin{array}{l}
-\lap u= \lambda u \text{ in }\Omega\\[4pt]
u\geq 0 \text{ is bounded.}
\end{array}\right.$$

We prove now that $u$ is continuous and vanishes on $\partial \Omega$. Observe that for any $x_1 \in \partial \Omega$, there exist a ball $B=B_{r_1}(y)$ such that $\overline{\Omega} \cap \overline{B}=\{x_1\}$ and, from Proposition \ref{eigenfunctionOutsideABall}, a positive radially symmetric eigenfunction $w$ associated to $\lambda$ in $\mathbb{H}^n \backslash B$. The set of maximum points of $w$ is a sphere $\partial B_{r_2}(y)$, where $r_2 > r_1$, and we can assume that $\max w = \max v_1$.
For $N$ large, $u_N$ is defined in $\Omega \cap B_{r_2}(y)$, since $\Omega_N = \Omega \cap B_N(x_0) \supset \Omega \cap B_{r_2}(y)$. Moreover, $w \ge u_N$ on $\partial (\Omega \cap B_{r_2}(y) )= \left( \partial \Omega \cap \overline{B_{r_2}(y)}\right)  \cup \left( \Omega \cap \partial B_{r_2}(y) \right)$, since $w\ge 0 = v_0=u_N$ on $\partial \Omega \subset \mathbb{H}^n \backslash B_{r_1}(y)$ and $ w=\max w \ge \max v_1 \ge v_1 > v_0=u_N$ on $\partial B_{r_2}(y)$.  
Therefore, $w \ge u_N$ in $\Omega \cap B_{r_2}(y)$, otherwise $u_N - w$ is a positive eigenfunction in some open subset $\Omega' \subset \Omega \cap B_{r_2}(y)$ that vanishes on $\partial \Omega'$, contradicting Corollary \ref{NoEigenfunctionInBoundedDomains}. Hence $0 \le u \le w$ in $\Omega \cap B_{r_2}(y)$, $w(x_1)=0$ and $u(x_1)=\lim u_N(x_1) = 0$. From the continuity of $w$, we have that $u$ is continuous and equals to zero at $x_1$ proving the statement. 

Finally, from \eqref{desigualdades} and $\lim_{x\to \infty}v_1(x)=0$, $u$ extends continuously to the asymptotic boundary of the domain, where it is equal to zero.
\end{proof}

\begin{proposition}
If $u$ is an eigenfunction associated to $\lambda \in (0,\lambda_1]$ in $\Omega \subset \mathbb{H}^n$
that can be extended continuously at $\partial_{\infty} \Omega$, then $\lim_{x\to \infty} u(x) = 0$.
\end{proposition}

\begin{proof}
Given $z_0 \in \partial_{\infty} \Omega$, let $(x_k)$ be a sequence in $\mathbb{H}^n$ such that $x_k \to z_0$. 
If $dist(x_k, \partial \Omega) \to 0$, then, using that $u=0$ on $\partial \Omega$, there exist a sequence of points $y_k \in \Omega$, close to the boundary, such that $u(y_k) \to 0$ and $dist(y_k,x_k) \to 0$. Hence, the existence of the limit implies that $\lim_{k\to \infty} u(x_k) =0$.

Therefore, if $\lim_{k\to \infty} u(x_k) =L \ne 0$, it follows that $dist(x_k, \partial \Omega) \not\to 0$. Thus, for some subsequence, say $(x_k)$, holds $dist(x_k,\partial \Omega) \ge r_0$, for some $r_0 >0$. Suppose, without loss of generality, that $L >0$. From the cone topology, given $\varepsilon > 0$ ($\varepsilon < L$), there is $k_0 \in \mathbb{N}$ such that 
$$0 < L-\varepsilon < u(x) < L + \varepsilon \quad {\rm  for} \quad x\in \overline{B_{r_0}(x_k)}$$ if $k \ge k_0$.
Hence $-\Delta u = \lambda u \ge \lambda (L - \varepsilon)$ in $B_{r_0}(x_k)$. For $k\ge k_0$, define 
$$ P_k(x)= (L-\varepsilon) + C \left(r_0^2 - (dist(x,x_k))^2\right),$$
where $C > 0$ is some suitable constant. Note that $P_k$ is radially symmetric with respect to $x_k$ and, expressing its Laplacian in radial coordinates with $r=dist(x,x_k)$, we have $$-\Delta P_k(r) = -\frac{d^2 P_k}{d\, r^2} - (n-1) \coth r \frac{dP_k}{d\, r} \le C[2 + (n-1) (\coth r_0) 2r_0 ],$$
for $r \le r_0$, since $[2 + (n-1) (\coth r) 2r ]$ is increasing. Then, for $C$ small not depending on $k$ and $\varepsilon$, we have $-\Delta P_k \le \lambda (L - \varepsilon)$.
Therefore, $$-\Delta u \ge -\Delta P_k  \quad {\rm in } \quad B_{r_0}(x_k) \quad {\rm and} \quad  u > L-\varepsilon \ge P_k \quad {\rm on } \quad \partial B_{r_0}(x_k).$$ Hence, from Lemma \ref{comparison II}, 
$u > P_k$ in $B_{r_0}(x_k)$ and, thus, $u(x_k) > P_k(x_k)= L-\varepsilon + Cr_0^2$. This contradicts $u (x_k) < L + \varepsilon$ for $\varepsilon < Cr_0^2/2$,
concluding the proof. 
\end{proof}

\noindent There exist positive bounded eigenfunctions that cannot be extended continuously at $\partial_{\infty} \Omega$. This is exemplified in the next proposition.

\begin{proposition}
Let $B$ be a horoball in $\mathbb{H}^n$, with boundary $H$.
Then, problem \eqref{eqProblemaPrincipal} with $U=\mathbb{H}^n\backslash B$
has a positive bounded solution that depends only on $d=dist(x,H)$. 
This solution extends continuously to zero at $\partial_{\infty}\mathbb{H}^n\backslash \partial_{\infty} H$ and 
it cannot be extended to $\partial_{\infty} H\cap\partial_{\infty}\mathbb{H}^n.$
\end{proposition}

\begin{proof}
According to \eqref{EDO horo exterior}, $\bar{u}$ is an eigenfunction associated to $\lambda \in (0,\lambda_1]$ on $\mathbb{H}^n\backslash B$ of the form $\bar{u}(x)=u(d(x))$ if and only if 
$$u''(d)+(n-1)u'(d)+\lambda u =0 \quad {\rm and} \quad u(0)=0,$$ which has the following solution
\begin{equation} u(d)= \left\{ \begin{array}{ll} C\,d \,e^{-\frac{(n-1)d}{2}} & \quad {\rm if} \quad \lambda = \lambda_1  \\[5pt] Ce^{-\frac{(n-1)d}{2}}\left( e^{\frac{ \sqrt{(n-1)^2 -4\lambda}}{2}d} - e^{-\frac{ \sqrt{(n-1)^2 -4\lambda}}{2}d} \right) & \quad {\rm if} \quad \lambda < \lambda_1 \end{array} \right.
\label{uBoundedEigenfunctionOutsideHoroball} \end{equation} where $C$ is any real constant. Then, for any $C>0$, $\bar{u}$ is positive and bounded in $\mathbb{H}^n\backslash B$.  To see that it extends continuously to $\partial_{\infty}\mathbb{H}^n\backslash \partial_{\infty} H,$ take a point $p \in \partial_{\infty}\mathbb{H}^n\backslash \partial_{\infty} H.$
Given $\epsilon >0,$ there is $d_0$ large enough such that $d\ge d_0$ implies
$u(d)<\epsilon.$ Consider $H_{d_0}$ the horosphere parallel to $H$ of distance $d_0$ from $H$ and call $B_{d_0}$ the closed horoball bounded by $H_{d_0}$.
The set $\overline{\mathbb{H}^n}\backslash B_{d_0}$ contains an open set around $p$ and is contained in $\{u<\epsilon\}.$
On the other hand, if $p\in \partial_{\infty} H\cap\partial_{\infty}\mathbb{H}^n,$ any open set containing $p$ intercepts all horospheres parallel to $H,$ which are the level sets of $u.$ Hence there is no continuous extension of $u$ at $p.$
\end{proof}

\

\

Departamento de Matem\'atica Pura e Aplicada

Universidade Federal do Rio Grande do Sul

Av. Bento Gon\c{c}alves 9500 - Pr\'edio 43111 

91509-900 Porto Alegre - RS - BRASIL

bonorino@mat.ufrgs.br

\

Departamento de Matem\'atica Pura e Aplicada

Universidade Federal do Rio Grande do Sul

Av. Bento Gon\c{c}alves 9500 - Pr\'edio 43111 

91509-900 Porto Alegre - RS - BRASIL

patricia.klaser@ufrgs.br

\end{document}